\documentclass[12 pt,reqno]{amsart}
\usepackage{amscd,amssymb,amsthm}
\usepackage[arrow,matrix]{xy}
\usepackage[displaymath]{lineno}
\usepackage{graphicx}
\usepackage{cite}
\usepackage{float}
\usepackage[super,sort&compress]{natbib}
\usepackage{cite}
\setcitestyle{citesep={,}}
\makeatletter
\renewcommand\@biblabel[1]{(#1)}
\makeatother
\oddsidemargin=0.3in \evensidemargin=0.3in \theoremstyle{plain}
\newtheorem{thm}[subsection]{Theorem}
\newtheorem{lem}[subsection]{Lemma}

\newtheorem{cor}[subsection]{Corollary}

\theoremstyle{definition}

\numberwithin{equation}{section} 

\begin{document}

\title [\texttt{A Novel/Old Modification of the First Zagreb Index}]{A Novel/Old Modification of the First Zagreb Index}
\author{Akbar Ali$^{1}$ and Nenad Trinajsti\'{c}$^{2}$}
 \address{$^{1}$Department of Mathematics, University of Management and Technology, Sialkot, Pakistan.}
 \email{akbarali.maths@gmail.com}
 \address{$^{2}$The Rugjer Bo\v{s}kovi\'{c} Institute, P. O. Box 180, HR-10002 Zagreb, Croatia.}
\email{trina@irb.hr}

\keywords{topological index, Zagreb indices, chemical tree, vertex connection number}

\begin{abstract}

In the paper [Gutman, I.; Trinajsti\'{c}, N. \textit{Chem. Phys. Lett.} \textbf{1972}, 17, 535], it was shown that total $\pi$-electron energy ($E$) of a molecule $M$ depends on the quantity $\sum_{v\in V(G)}d_{v}^{2}$ (nowadays known as the \textit{first Zagreb index}), where $G$ is the graph corresponding to $M$, $V(G)$ is the vertex set of $G$ and $d_{v}$ is degree of the vertex $v$. In the same paper, the graph invariant $\sum_{v\in V(G)}d_{v}\tau_{v}$ (where $\tau_{v}$ is the connection number of $v$, that is the number of vertices at distance 2 from $v$) was also proved to influence $E$, but this invariant was never restudied explicitly. We call it \textit{modified first Zagreb connection index} and denote it by $ZC_{1}^{*}$. In this paper, we characterize the extremal elements with respect to the graph invariant $ZC_{1}^{*}$ among the collection of all $n$-vertex chemical trees.

\end{abstract}

\maketitle

\section{Introduction}

It is well known fact that chemical compounds can be represented by graphs (known as molecular graphs) in which vertices correspond to the atoms while edges represent the covalent bonds between atoms \cite{trn,gutman1}. In theoretical chemistry, the physicochemical properties of chemical compounds are often modeled by the topological indices \cite{t1,f1}. Topological indices are numerical quantities of molecular graph, which are invariant under graph isomorphism \cite{f1}.  In the paper \cite{g3}, it was shown that the following topological indices appears in an approximate formula for total $\pi$-electron energy ($E$) of a molecule $M$:
$$M_{1}(G)=\sum_{v\in V(G)}d_{v}^{2} \ , \ \ \ ZC_{1}^{*}(G)=\sum_{v\in V(G)}d_{v}\tau_{v}$$
where $G$ is the graph corresponding to the molecule $M$, $V(G)$ is the vertex set of $G$, $d_{v}$ is degree of the vertex $v$ and $\tau_{v}$ is the connection number of $v$ (that is, the number of vertices at distance 2 from $v$). The topological index $M_{1}$ is known as the \textit{first Zagreb index}. We call the topological index $ZC_{1}^{*}$ as \textit{modified first Zagreb connection index}. The following topological index is called \textit{second Zagreb index}: \cite{g3,gutman75}
$$M_{2}(G)=\sum_{uv\in E(G)}d_{u}d_{v},$$
where $E(G)$ is the edge set of the graph $G$ and $uv$ is the edge between the vertices $u,v$.
The first Zagreb index and second Zagreb index are among the oldest and most studied topological indices. More than a hundred papers have been devoted to these Zagreb indices, for example see the reviews \cite{Nic03, Das04, Gut04} published on the occasion of their 30th anniversary, recent surveys \cite{g1,Gut14}, very recent papers \cite{Habi16,Reti16,Bor16,wang16,lee16,das16,deng16} and related references cited therein. In contrast, the modified first Zagreb connection index did not attract attention in any of the numerous publications on topological indices till 2016. The main purpose of the present study is to gather some basic properties of the modified first Zagreb connection index, and-especially-to characterize the extremal elements with respect to the aforementioned topological index among the collection of all $n$-vertex chemical trees.

\section{Some Observations on the Modified First Zagreb Connection Index}

In this section, some basic properties of the modified first Zagreb connection index are established. As usual, the path graph, star graph and complete graph on $n$ vertices will be denoted by $P_{n}$, $S_{n}$ and $K_{n}$, respectively. Undefined notations and terminologies from (chemical) graph theory can be found in the books \cite{Ha69,trn,bon76}. Suppose that the graph $G$ has $n_{0}$ vertices with degree zero and let $V_{0}(G)$ be the set of these vertices. In the Reference \cite{Doslic11}, the following identity was derived:
\begin{equation}\label{Eq113}
\sum_{u\in V(G)\setminus V_{0}(G)}d_{u}f(d_{u})=\sum_{uv\in E(G)}(f(d_{u})+f(d_{v})),
\end{equation}
where $f$ is a function defined on the set of vertex degrees of the graph $G$. The proof technique used in establishing the identity (\ref{Eq113}) also works for deriving the following more general identity:
\begin{equation}\label{Eq114}
\sum_{u\in V(G)\setminus V_{0}(G)}d_{u}g(u)=\sum_{uv\in E(G)}(g(u)+g(v)),
\end{equation}
where $g$ is a function defined on the vertex set $V(G)$. Hence, the modified first Zagreb connection index can be rewritten in the following form by setting $g(u)=\tau_{u}$ in the identity (\ref{Eq114}):
\begin{equation}\label{Eq115}
ZC_{1}^{*}(G)=\sum_{uv\in E(G)}(\tau_{u}+\tau_{v}).
\end{equation}

In 2008, Yamaguchi \cite{Yamaguchi} derived a lower bound (given in the following theorem) for the topological index $M_{2}$ in which the quantity $\sum_{v\in V(G)}d_{v}\tau_{v}$ also appears:

\begin{thm}\label{t555}
If $G$ is a connected graph, then
\begin{equation}\label{Eq111}
M_{2}(G)\geq\frac{1}{2}\left(\sum_{v\in V(G)}d_{v}(d_{v}+\tau_{v})\right),
\end{equation}
with equality if and only if $G$ is a triangle- and quadrangle-free
graph.
\end{thm}
Hence, if the graph under consideration is triangle- and quadrangle-free then the modified first Zagreb connection index can be written as the linear combination of the topological indices $M_{1}$ and $M_{2}$.
\begin{cor}\label{c555}
If $G$ is a triangle- and quadrangle-free
graph, then
\begin{equation}\label{Eq1110}
ZC_{1}^{*}(G)=2M_{2}(G)-M_{1}(G).
\end{equation}
\end{cor}

Here, it needs to be mentioned that Xu \textit{et al.} \cite{Xu-14} considered the linear combination $2M_{2}(G)-M_{1}(G)$ for any connected graph $G$. As there are many (lower and upper) bounds as well as relations for the topological indices $M_{1}$ and $M_{2}$ in the literature. Hence, several bounds for the topological index $ZC_{1}^{*}$ can be easily established using Eq. (\ref{Eq1110}) in case of triangle- and quadrangle-free graphs. For instance, we derive an upper bound for the topological index $ZC_{1}^{*}$ using the following result which was reported in the references \cite{Das04,Das09}:

\begin{thm}\label{t666}
If $G$ is an $n$-vertex graph with size $m$ and minimum vertex degree $\delta$ then
\[M_{2}(G)\leq2m^{2}-(n-1)m\delta+\frac{1}{2}(\delta-1)M_{1}(G)\]
with equality if and only if $G\cong S_{n}$ or $G\cong K_{n}$.
\end{thm}

The upper bound for $ZC_{1}^{*}$, given in the following corollary, is a direct consequence of Theorem \ref{t666}:

\begin{cor}\label{cor666}
If $G$ is a triangle- and quadrangle-free graph with order $n$, size $m$ and minimum vertex degree $\delta$ then
\[ZC_{1}^{*}(G)\leq4m^{2}-2(n-1)m\delta+(\delta-2)M_{1}(G)\]
with equality if and only if $G\cong S_{n}$.
\end{cor}

\begin{figure}[p]
  \centering
   \includegraphics[width=6.2in, height=3in]{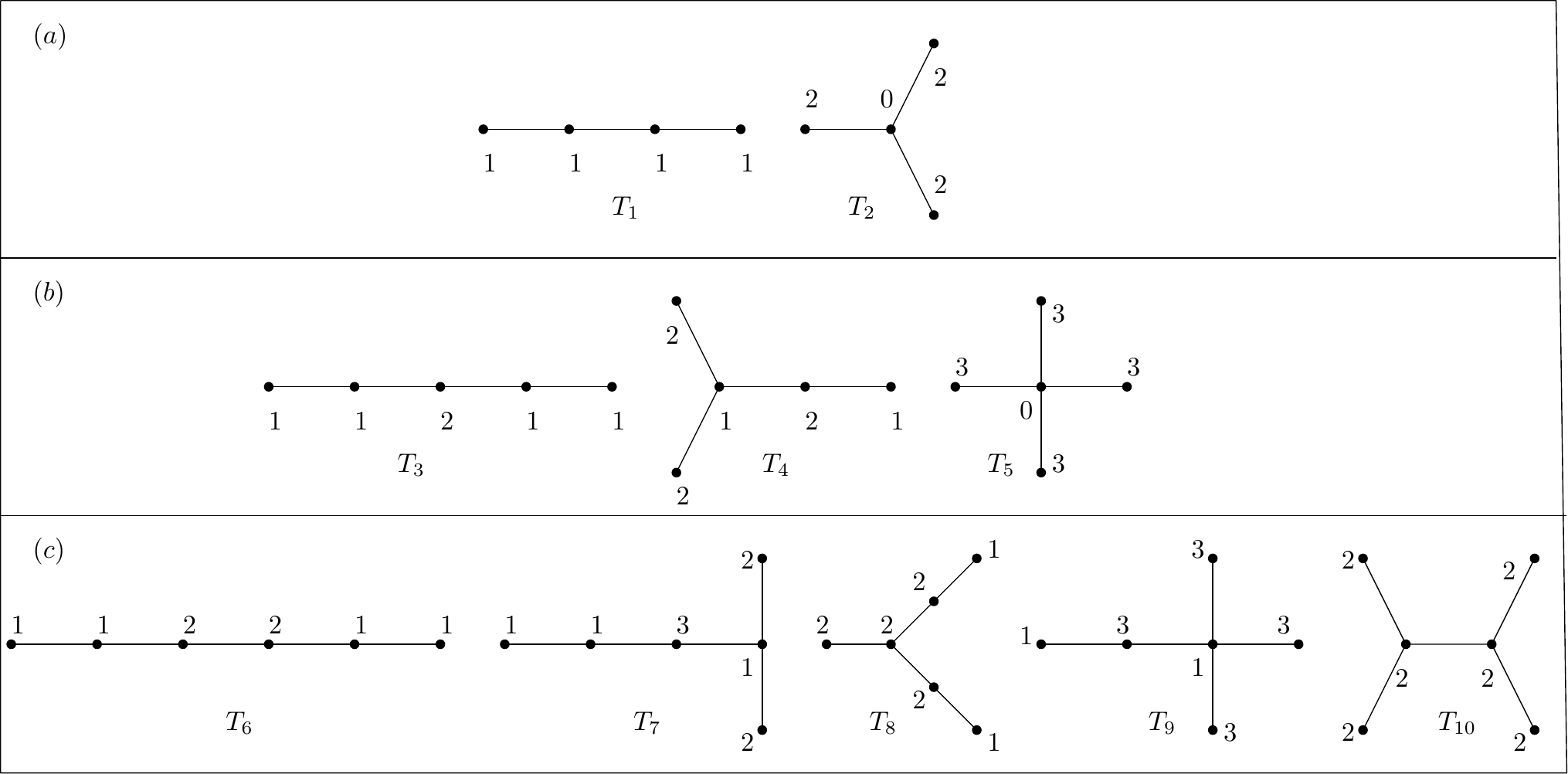}
    \caption{All the non-isomorphic chemical trees (together with vertex connection numbers) on (a) four vertices (b) five vertices (c) six vertices.}
    \label{f1}
     \end{figure}

\section{Modified First Zagreb Connection Index of chemical trees}

Denote by $\mathbb{CT}_{n}$ the collection of all $n$-vertex chemical trees. This section is devoted to characterize the extremal elements with respect to the topological index $ZC_{1}^{*}$ among the collection $\mathbb{CT}_{n}$. Note that the collection $\mathbb{CT}_{n}$ consist of only a single element for $n=1,2,3$. All the non-isomorphic members of of the collections $\mathbb{CT}_{4}$, $\mathbb{CT}_{5}$ and $\mathbb{CT}_{6}$ (together with vertex connection numbers) are depicted in Figure \ref{f1} and their $ZC_{1}^{*}$ values are given in Table \ref{table111}. Note that both non-isomorphic trees on 4 vertices have the same $ZC_{1}^{*}$ value. Hence, the before said problem concerning extremal chemical trees make sense only for $n\geq5$. Firstly, we characterize the chemical tree having minimum $ZC_{1}^{*}$ value among the aforementioned collection for $n\geq5$. For a vertex $u\in V(G)$, denote by $N(u)$ (the neighborhood of $u$) the set of all vertices adjacent with $u$. A vertex having degree 1 is called pendent vertex. A pendent vertex adjacent with a vertex having degree greater than 2 is called star-type pendent vertex.

\begin{table}[p]
\begin{tabular}{|c|c|c|c|c|c|c|c|c|c|c|c|c|c|c|c|c|c|c|c|c|}\hline%
Chemical Tree $T_{i}$ shown in Figure \ref{f1}  &$T_{1}$&$T_{2}$&$T_{3}$&$T_{4}$&$T_{5}$&$T_{6}$&$T_{7}$&$T_{8}$&$T_{9}$&$T_{10}$\\\hline%
The $ZC_{1}^{*}$ Value of $T_{i}$&6&6&10&12&12&14&16&18&20&20\\\hline
\end{tabular}
\caption{The $ZC_{1}^{*}$ values of the chemical trees depicted in Figure \ref{f1}.}
\label{table111}
\end{table}

\begin{thm}\label{t00}
For $n\geq5$, the path graph has minimum $ZC_{1}^{*}$ value among all the members of $\mathbb{CT}_{n}$.
\end{thm}

\begin{proof}
Simple calculations yield $ZC_{1}^{*}(P_{n})=4n-10$. The result will be proved by induction on $n$. For $n=5$, there are only three non-isomorphic trees and hence the conclusion can be easily verified. Suppose that the result holds for all trees of order $\leq n-1$ where $n\geq6$. Let $T_{n}$ be an $n$-vertex tree and $u$ be its pendent vertex adjacent with the vertex $v$. Set $d_{v}=x$ and $N(v)=\{u=u_{0},u_{1},u_{2},...,u_{r-1},u_{r},...,u_{x-1}\}$ where $d_{u_{i}}=1$ for $0\leq i\leq r-1$ and $d_{u_{i}}\geq2$ for $r\leq i\leq x-1$. As $T_{n}$ is different from the star graph $S_{n}$, so $\tau_{v}\geq1$. Let $T_{n-1}$ be the tree obtained from $T_{n}$ by removing the vertex $u$. Then
$$ZC_{1}^{*}(T_{n})=ZC_{1}^{*}(T_{n-1})+\displaystyle\sum_{i=1}^{x-1}d_{u_{i}}\tau_{u_{i}}+d_{v}\tau_{v}+\tau_{u}
 -\displaystyle\sum_{i=1}^{x-1}d_{u_{i}}(\tau_{u_{i}}-1)-(d_{v}-1)\tau_{v}$$
which is equivalent to
\begin{equation}\label{eq00}
ZC_{1}^{*}(T_{n})=ZC_{1}^{*}(T_{n-1})+\displaystyle\sum_{i=1}^{x-1}d_{u_{i}}+d_{v}+\tau_{v}-1
\end{equation}
Bearing in mind the facts $\tau_{v}\geq1$, $d_{v}\geq2$, $\sum_{i=1}^{x-1}d_{u_{i}}\geq2$ and inductive hypothesis, from eq. (\ref{eq00}) we have
$$ZC_{1}^{*}(T_{n})\geq ZC_{1}^{*}(T_{n-1})+4\geq 4(n-1)-10+4\geq ZC_{1}^{*}(P_{n}).$$
Observe that the equality $ZC_{1}^{*}(T_{n})=ZC_{1}^{*}(P_{n})$ holds if and only if $\tau_{v}=1$, $d_{v}=2$, $d_{u_{1}}=2$ and $T_{n-1}\cong P_{n-1}$. This completes the proof.
\end{proof}

It can be easily noted that the proof of Theorem \ref{t00} remains valid if we replace the collection $\mathbb{CT}_{n}$ with the collection of all $n$-vertex general trees different from the star graph $S_{n}$, where $n\geq5$. Moreover, $ZC_{1}^{*}(P_{n})=4n-10<ZC_{1}^{*}(S_{n})=(n-1)(n-2)$ for all $n\geq5$. Hence, one has:

\begin{cor}\label{t6}
For $n\geq5$, the path has minimum $ZC_{1}^{*}$ value among all $n$-vertex trees.
\end{cor}

Now, we prove some lemmas which will be used to characterize the trees with maximum $ZC_{1}^{*}$ value among the collection $\mathbb{CT}_{n}$. In the remaining part of this section, we will use the alternative formula of $ZC_{1}^{*}$, given in Eq. (\ref{Eq1110}).

\begin{lem}\label{lem1}
For $n\geq5$, if the tree $T^{*}\in\mathbb{CT}_{n}$ has the maximum $ZC_{1}^{*}$ value then $T^{*}$ contains at most one vertex of degree 2.
\end{lem}

\begin{proof}
Suppose to the contrary that $T^{*}$ contains more than one vertex of degree 2. Let $u,v\in V(T^{*})$ such that $d_{u}=d_{v}=2$. Let $N(u)=\{u_{1},u_{2}\}$ and $N(v)=\{v_{1},v_{2}\}$. Suppose that the unique path connecting the vertices $u$ and $v$ contains the vertices $u_{2},v_{2}$. Without loss of generality, we may assume that $d_{u_{1}}+d_{u_{2}}\leq d_{v_{1}}+d_{v_{2}}$. Let $T^{(1)}$ be the tree obtained from $T^{*}$ by removing the edge $u_{1}u$ and adding the edge $u_{1}v$. There are two cases.

\textit{Case 1.} The vertices $u$ and $v$ are adjacent. That is, $u_{2}=v$ and $v_{2}=u$.\\
We calculate the value of the difference $ZC_{1}^{*}(T^{*})-ZC_{1}^{*}(T^{(1)})$. If the vertex $u_{1}$ is pendent then the condition $n\geq5$ guaranties that $d_{v_{1}}\geq2$. Hence, whether the vertex $u_{1}$ is pendent or not, in both cases we have
\[ZC_{1}^{*}(T^{*})-ZC_{1}^{*}(T^{(1)})=2(2-d_{u_{1}}-d_{v_{1}})<0,\]
which is a contradiction the definition of $T^{*}$. \\
\textit{Case 2.} The vertices $u$ and $v$ are not adjacent.\\
Bearing in mind the inequality $d_{u_{1}}+d_{u_{2}}\leq d_{v_{1}}+d_{v_{2}}$, we have
\begin{equation*}
ZC_{1}^{*}(T^{*})-ZC_{1}^{*}(T^{(1)})=2(d_{u_{2}}-d_{u_{1}}-d_{v_{1}}-d_{v_{2}}+1)<0,
\end{equation*}
which contradicts the maximality of $ZC_{1}^{*}(T^{*})$.\\
In both cases, contradiction arises. Hence, $T^{*}$ contains at most one vertex of degree 2.
\end{proof}

\begin{lem}\label{lem-1b}
For $n\geq5$, if the tree $T^{*}\in\mathbb{CT}_{n}$ has the maximum $ZC_{1}^{*}$ value and $T^{*}$ contains a vertex $u$ of degree 2 then one of the neighbors of $u$ is pendent.
\end{lem}

\begin{proof}
Suppose to the contrary that both the neighbors of $u$, say $u_{1}$ and $u_{2}$, are non-pendent. Then, Lemma \ref{lem1} suggests that both of the vertices $u_{1},u_{2}$ must have degree at least 3. Observe that all the pendent vertices of $T^{*}$ are star-type.\\

\textit{Case 1.} Neither of the vertices $u_{1},u_{2}$ has a pendent neighbor.\\
Let $w\in V(T^{*})$ be a pendent vertex adjacent with the vertex $t\in V(T^{*})$. Suppose that $T^{(1)}$ is the tree obtained from $T^{*}$ by removing the edges $u_{1}u,u_{2}u,wt$ and adding the edges $u_{1}u_{2},wu,ut$. Observe that both the trees $T^{(1)}$ and $T^{*}$
have same degree sequence. Note that $2d_{u_{1}}+2d_{u_{2}}-d_{u_{1}}d_{u_{2}}\leq3$, which implies that
\[ZC_{1}^{*}(T^{*})-ZC_{1}^{*}(T^{(1)})=2(2d_{u_{1}}+2d_{u_{2}}-d_{u_{1}}d_{u_{2}}-2-d_{t})\leq2(1-d_{t})<0.\]
This contradicts the maximality of $ZC_{1}^{*}(T^{*})$.

\textit{Case 2.} At least one of the vertices $u_{1},u_{2}$ has a pendent neighbor.\\
Without loss of generality, assume that $u_{1}$ is adjacent with a pendent vertex $u_{11}\in V(T^{*})$. Let $T^{(2)}$ be the tree obtained from $T^{*}$ by removing the edges $u_{1}u,u_{2}u$ and adding the edges $u_{1}u_{2},u_{11}u$. Bearing in mind the inequality $d_{u_{1}}+2d_{u_{2}}-d_{u_{1}}d_{u_{2}}\leq0$, one has
\[ZC_{1}^{*}(T^{*})-ZC_{1}^{*}(T^{(2)})=2(d_{u_{1}}+2d_{u_{2}}-d_{u_{1}}d_{u_{2}}-2)<0,\]
which is again a contradiction. This completes the proof.

\end{proof}

\begin{lem}\label{lem2}
For $n\geq7$, if the tree $T^{*}\in\mathbb{CT}_{n}$ has the maximum $ZC_{1}^{*}$ value then $T^{*}$ contains at most one vertex of degree 3.
\end{lem}

\begin{proof}
Suppose, contrarily, that $T^{*}$ contains at least two vertices of degree 3. Let $u,v\in V(T^{*})$ such that $d_{u}=d_{v}=3$. Let $N(u)=\{u_{1},u_{2},u_{3}\}$ and $N(v)=\{v_{1},v_{2},v_{3}\}$. Suppose that the unique path connecting the vertices $u$ and $v$ contains the vertices $u_{3},v_{3}$. Without loss of generality, we may also assume that $d_{u_{1}}+d_{u_{2}}+d_{u_{3}}\leq d_{v_{1}}+d_{v_{2}}+d_{v_{3}}$ and $d_{u_{1}}\geq d_{u_{2}}$. Let $T^{(1)}$ be the tree obtained from $T^{*}$ by removing the edge $u_{1}u$ and adding the edge $u_{1}v$. There are two cases.

\textit{Case 1.} The vertices $u$ and $v$ are adjacent. That is, $u_{3}=v$ and $v_{3}=u$.\\
If both the vertices $u_{1}$ and $u_{2}$ are pendent then the assumption $n\geq7$ forces that at least one of the vertices $v_{1}$ and $v_{2}$ must be non-pendent and hence
\begin{eqnarray*}
ZC_{1}^{*}(T^{*})-ZC_{1}^{*}(T^{(1)})&=& 2(2+d_{u_{2}}-d_{u_{1}}-d_{v_{1}}-d_{v_{2}})<0,
\end{eqnarray*}
which is a contradiction.\\
If at least one of the vertices $u_{1},u_{2}$ is non-pendent then $d_{u_{1}}\geq2$ (because $d_{u_{1}}\geq d_{u_{2}}$) and hence we have
\begin{eqnarray*}
ZC_{1}^{*}(T^{*})-ZC_{1}^{*}(T^{(1)})&=& 2(2+d_{u_{2}}-d_{u_{1}}-d_{v_{1}}-d_{v_{2}})\\
&\leq&2(2+d_{u_{1}}-d_{u_{2}}-d_{u_{1}}-d_{u_{2}})=4(1-d_{u_{1}})<0,
\end{eqnarray*}
again a contradiction.\\
\textit{Case 2.} The vertices $u$ and $v$ are not adjacent.\\
In this case we have:
\begin{equation}\label{Eq117}
ZC_{1}^{*}(T^{*})-ZC_{1}^{*}(T^{(1)})= 2(1-d_{u_{1}}+d_{u_{2}}+d_{u_{3}}-d_{v_{1}}-d_{v_{2}}-d_{v_{3}}).
\end{equation}
From the inequality $d_{u_{1}}+d_{u_{2}}+d_{u_{3}}\leq d_{v_{1}}+d_{v_{2}}+d_{v_{3}}$ and eq. \ref{Eq117} it follows that
$$ZC_{1}^{*}(T^{*})-ZC_{1}^{*}(T^{(1)})\leq2(1-2d_{u_{1}})<0,$$
which contradicts the maximality of $ZC_{1}^{*}(T^{*})$.\\
In both cases, contradiction arises. Hence, $T^{*}$ contains at most one vertex of degree 3.
\end{proof}

\begin{lem}\label{lem3}
For $n\geq7$,
if the tree $T^{*}\in\mathbb{CT}_{n}$ has the maximum $ZC_{1}^{*}$ value then $T^{*}$ does not contain vertices of degree 2 and 3 simultaneously.
\end{lem}

\begin{proof}
Suppose to the contrary that $u,v\in V(T^{*})$ such that $d_{u}=2$ and $d_{v}=3$. Let $N(u)=\{u_{1},u_{2}\}$ and $N(v)=\{v_{1},v_{2},v_{3}\}$. Suppose that the unique path connecting the vertices $u$ and $v$ contains the vertices $u_{2},v_{3}$. Let $T^{(1)}$ be the tree obtained from $T^{*}$ by removing the edge $u_{1}u$ and adding the edge $u_{1}v$. There are two cases.\\

\textit{Case 1.} The vertices $u$ and $v$ are adjacent. That is, $u_{2}=v$ and $v_{3}=u$.\\
Note that the assumption $n\geq7$ implies that at least one of the vertices $u_{1}$, $v_{1}$, $v_{2}$ must be non-pendent and hence
\begin{eqnarray*}
ZC_{1}^{*}(T^{*})-ZC_{1}^{*}(T^{(1)})&=& 2(4-2d_{u_{1}}-d_{v_{1}}-d_{v_{2}})<0,
\end{eqnarray*}
which is a contradiction.\\

\textit{Case 2.} The vertices $u$ and $v$ are not adjacent.\\
By the virtue of Lemma \ref{lem1} and Lemma \ref{lem2}, we have $d_{v_{3}}=4$ and hence
\begin{equation*}
ZC_{1}^{*}(T^{*})-ZC_{1}^{*}(T^{(1)})= 2(2+d_{u_{2}}-2d_{u_{1}}-d_{v_{1}}-d_{v_{2}}-d_{v_{3}})<0,
\end{equation*}
which contradicts the maximality of $ZC_{1}^{*}(T^{*})$. This completes the proof.
\end{proof}

\begin{lem}\label{lem4}
For $n\geq7$, let the tree $T^{*}\in\mathbb{CT}_{n}$ has the maximum $ZC_{1}^{*}$ value. If $T^{*}$ contains a vertex $u$ of degree 3 then $u$ has two pendent neighbors.
\end{lem}

\begin{proof}
Contrarily, suppose that $u$ has at least two non-pendent neighbors. Let $N(u)=\{u_{1},u_{2},u_{3}\}$. Without loss of generality, we may assume that $u_{1},u_{2}$ are non-pendent. Then, due to Lemmas \ref{lem1} - \ref{lem3}, both the vertices $u_{1},u_{2}$ have degree 4 and there must exist a vertex $v\in V(T^{*})$ of degree 4 such that three neighbors of $v$ are pendent. Let $w$ be a pendent neighbor of $v$ and suppose that $T^{(1)}$ is the tree obtained from $T^{*}$ by removing the edge $vw$ and adding the edge $uw$.

\textit{Case 1.} The vertex $v$ doest not belong to the set $N(u)$.\\
Let $v_{1}$ be the unique non-pendent neighbor of $v$. Then, simple calculations yield:
\[ZC_{1}^{*}(T^{*})-ZC_{1}^{*}(T^{(1)})= 2(d_{v_{1}}-d_{u_{3}}-6)<0,\]
which is a contradiction.

\textit{Case 2.} The vertex $v$ belongs to the set $N(u)$.\\
Without loss of generality, we may assume that $v=u_{1}$. Then, we again have a contradiction as follows:
\[ZC_{1}^{*}(T^{*})-ZC_{1}^{*}(T^{(1)})= -2d_{u_{3}}-4<0.\]
This completes the proof.
\end{proof}

For $n\geq9$ and $n\equiv0$ (mod 3), let $\mathbb{CT}_{n}^{(0)}$ be the collection of those $n$-vertex chemical trees in which one vertex has degree 2 whose one neighbor is pendent, and every other vertex has degree 1 or 4. For $n\geq7$ and $n\equiv1$ (mod 3), let $\mathbb{CT}_{n}^{(1)}$ be the collection of those $n$-vertex chemical trees in which one vertex has degree 3 whose two neighbors are pendent, and every other vertex has degree 4 or 1. For $n\geq8$ and $n\equiv2$ (mod 3), let $\mathbb{CT}_{n}^{(2)}$ be the collection of those $n$-vertex chemical trees which consists of only vertices of degree 1 and 4. Let us take $\mathbb{CT}_{n}^{*}=\mathbb{CT}_{n}^{(0)}\cup\mathbb{CT}_{n}^{(1)}\cup\mathbb{CT}_{n}^{(2)}$.

\begin{thm}\label{thm-1111}
For $n\geq7$, the tree $T^{*}\in\mathbb{CT}_{n}$ has maximum $ZC_{1}^{*}$ value if and only if $T^{*}\in\mathbb{CT}_{n}^{*}$.
\end{thm}

\begin{proof}
The result follows from Lemmas \ref{lem1} - \ref{lem4} and definition of the collection $\mathbb{CT}_{n}^{*}$.
\end{proof}

It is interesting to see that the $n$-vertex chemical trees having maximum $M_{2}$ value \cite{vuki14} and maximum $ZC_{1}^{*}$ value among the collection $\mathbb{CT}_{n}$ are same for $n\geq7$ (it is not true for general trees). Denote by $x_{a,b}$ the number of edges in  the graph $G$ connecting the vertices of degrees $a$ and $b$. Let $n_{a}$ be the number of vertices of degree $a$ in the graph $G$. The following system of equations holds for all trees of the collection $\mathbb{CT}_{n}$.
\begin{equation}\label{Eq555}
\sum_{i=1}^{4}n_{i}=n
\end{equation}
\begin{equation}\label{Eq556}
\sum_{i=1}^{4}i\times n_{i}=2(n-1)
\end{equation}
\begin{equation}\label{Eq557}
\sum_{ \substack{ 1\leq i\leq 4, \\
         i\neq j}}x_{j,i}+2x_{j,j}=j\times n_{j} \text{ \ } ; \text{ \ \ \ $j=1,2,3,4$.}
\end{equation}
Eliminating $n_{1}$ from eq. (\ref{Eq555}) and eq. (\ref{Eq556}):
\begin{equation}\label{Eq558}
\sum_{i=2}^{4}(i-1)\times n_{i}=n-2
\end{equation}

\begin{cor}\label{cor-1111}
For $n\geq7$, let $T$ be any member of $\mathbb{CT}_{n}$. Then
\[ZC_{1}^{*}(T)\leq
\begin{cases}
10(n-4)  & \text{if $n\equiv0$ (mod 3) or $n\equiv1$ (mod 3),} \\
2(5n-19) & \text{otherwise.}
\end{cases}\]
The equality sign in the first inequality holds if and only if $n\equiv0$ (mod 3) and $T\in\mathbb{CT}_{n}^{(0)}$, or $n\equiv1$ (mod 3) and $T\in\mathbb{CT}_{n}^{(1)}$. The equality sign in the second inequality holds if and only if $n\equiv2$ (mod 3) and $T\in\mathbb{CT}_{n}^{(2)}$.
\end{cor}

\begin{proof}
Let $T^{*}\in\mathbb{CT}_{n}^{*}$ and $T\not\in\mathbb{CT}_{n}^{*}$. Then, by virtue of Theorem \ref{thm-1111}, we have $ZC_{1}^{*}(T)< ZC_{1}^{*}(T^{*})$. Hence, to obtain the desired result, it is enough to calculate $ZC_{1}^{*}$ value of $T^{*}$.
If $n\equiv0$ (mod 3) then $n=3k$ where $k\geq3$. From eq. (\ref{Eq558}), it follows that $n_{2}+2n_{3}\equiv1$ (mod 3) which implies that $n_{2}=1,n_{3}=0$. Hence, from eq. (\ref{Eq555}) and eq. (\ref{Eq556}) we have $n_{1}=2k,n_{4}=k-1$. Also, note that $x_{1,2}=x_{2,4}=1$. Now, from system (\ref{Eq557}) we have $x_{1,4}=2k-1,x_{4,4}=k-2$ and hence $ZC_{1}^{*}(T^{*})=10(n-4)$. In a similar way,
we have:
\[ZC_{1}^{*}(T^{*})=
\begin{cases}
10(n-4)  & \text{if $n\equiv1$ (mod 3),} \\
2(5n-19) & \text{if $n\equiv2$ (mod 3).}
\end{cases}\]
From the definitions of the collections $\mathbb{CT}_{n}^{(0)},\mathbb{CT}_{n}^{(1)},\mathbb{CT}_{n}^{(2)},\mathbb{CT}_{n}^{*}$ and from Theorem \ref{thm-1111}, the desired result follows.

\end{proof}

For $n=5,6,$ the trees having maximum $ZC_{1}^{*}$ value among all the members of $\mathbb{CT}_{n}$ can be easily identified from Figure \ref{f1}, using Table \ref{table111}.

\section{Concluding Remarks}

We have fully characterized the extremal chemical trees with fixed number of vertices for the topological index $ZC_{1}^{*}$, which is occurred in an approximate formula for total $\pi$-electron energy, communicated in 1972.

The first Zagreb index $M_{1}$ can also be rewritten as
\begin{equation}\label{Eq1115}
M_{1}(G)=\sum_{uv\in E(G)}(d_{u}+d_{v}).
\end{equation}
So, if we replace vertex degree by vertex connection number in eq. (\ref{Eq1115}) we get the topological index $ZC_{1}^{*}$. Hence, it is natural to consider the following connection-number versions of the Zagreb indices $\sum_{v\in V(G)}d_{v}^{2}$ and $\sum_{uv\in E(G)}d_{u}d_{v}$:
$$ZC_{1}(G)=\sum_{v\in V(G)}\tau_{v}^{2} \ \ \ \text{and} \ \ \ ZC_{2}(G)=\sum_{uv\in E(G)}\tau_{u}\tau_{v} \ ,$$
respectively. We call $ZC_{1}$ as the first Zagreb connection index and $ZC_{2}$ as the second Zagreb connection index. Moreover, the following bond incident connection-number (BIC) indices may be considered as a generalization of the Zagreb connection indices:
\begin{equation}\label{z}
BIC(G)=\displaystyle\sum_{0\leq a\leq b\leq n-2}y_{a,b}(G).\phi_{a,b}
\end{equation}
where $y_{a,b}(G)$ is the number of edges in $G$ connecting the vertices with connection numbers $a$ and $b$, and $\phi_{a,b}$ is a non-negative real valued (symmetric) function which depends on $a$ and $b$.

\end{document}